\documentclass[12pt]{amsart}
\usepackage{color}
\usepackage[all]{xy}
\usepackage{amscd,amsmath,latexsym,amssymb}
\usepackage[mathcal]{euscript}

\date{June 4, 2019}

\calclayout
\newtheorem{dummy}{anything}[section]
\newtheorem{theorem}[dummy]{Theorem}
\newtheorem*{thma}{Theorem A}
\newtheorem*{thmb}{Theorem B}

\newtheorem{lemma}[dummy]{Lemma}
\newtheorem{proposition}[dummy]{Proposition}
\newtheorem{corollary}[dummy]{Corollary}
\newtheorem{conjecture}[dummy]{Conjecture}
\theoremstyle{definition}%%Change Theoremstyle
\newtheorem{definition}[dummy]{Definition}
  \newtheorem{example}[dummy]{Example}
  
  \newtheorem{remark}[dummy]{Remark}
   \newtheorem*{Remark}{Remark}
    
    \newtheorem*{question}{Question}
  \newtheorem*{acknowledgement}{Acknowledgement}

\newcommand
{\eqncount}{\setcounter{equation}{\value{dummy}}%
\addtocounter{dummy}{1}}
\newcommand{\cE}{\mathcal E}

\newcommand{\bZ}{\mathbb Z}

\newcommand{\bP}{\mathbf P}

\newcommand{\RP}{\mathbb R\bP}

\newcommand{\bS}{\mathbb S}

\newcommand{\wH}{\widehat H}
\DeclareMathOperator{\Hom}{Hom}

 \DeclareMathOperator{\Ext}{Ext}
\DeclareMathOperator{\Tor }{Tor}
 
\DeclareMathOperator{\Res}{Res}

\DeclareMathOperator{\coker}{coker}

 \DeclareMathOperator{\Syl}{Syl}

\newcommand{\cy}[1]{\bZ/{#1}\bZ}
\newcommand{\la}{\langle}
\newcommand{\ra}{\rangle}

\def\bZp{\bZ_{(p)}}

\newcommand{\mmatrix}[4]{\bigg (\hskip-4pt\vcenter
{\xymatrix@C-2pc@R-2pc{#1&#2\\#3&#4} }\hskip-2pt
\bigg )}

\newcommand{\ZG}{\bZ G}

\begin{document}

\title[Free Finite Group Actions on Rational Homology 3--Spheres]
{Free Finite Group Actions on Rational Homology 3--Spheres}

\author{Alejandro Adem}
\address{\vbox{\hbox{Department of Mathematics, The University of British Columbia} \hbox{Vancouver, British Columbia, Canada}}}
\email{adem@math.ubc.ca}

\author{Ian Hambleton}
\address{\vbox{\hbox{Department of Mathematics \& Statistics, McMaster University}\hbox{Hamilton, Ontario L8S 4K1, Canada}}}

\email{hambleton@mcmaster.ca }

\thanks{Both researchers were partially supported by NSERC Discovery Grants}

\begin{abstract}
We use methods from the cohomology of groups to describe the finite groups which can act freely and 
homologically trivially on closed 3--manifolds which are rational homology spheres. 
\end{abstract}

\maketitle

\section{Introduction}
Cooper and Long \cite{Cooper:2000} have shown that every finite group can act
freely and smoothly on some closed, oriented $3$-manifold
$M$ with the rational homology of the $3$-sphere (for brevity
we shall call such an object \textsl{a rational homology 3-sphere}). 
However, under the natural condition that the action must induce the identity 
on the integral homology of $M$,  new group theoretic restrictions arise. 
In this article, we apply group cohomology to establish necessary conditions 
for such \emph{homologically trivial} actions and use this information to construct 
some new examples.  

If a finite group $G$ acts freely on an integral homology $3$-sphere, then $G$ must have periodic cohomology of period two or four. Those which act freely on the standard $3$-sphere are the finite fundamental groups of closed 
$3$-manifolds. A basic example of a free action on a rational homology sphere is given by
the free action of $\cy{2} \times \cy{2}$ on the real projective space $\RP^3$, which arises from the free action
of the  quaternion group $Q(8)$ on the 3--sphere.
 %Associated to the quotient $\bS^3/Q(8)$ there is a 
This can be seen by considering the central extension 
$$1\to \cy{2} \to Q(8) \to \cy{2} \times\cy{2}\to 1. $$ 
We will show that all finite groups acting freely
and homologically trivially on rational homology spheres can be described in this way. 

If $\pi$ denotes a 
finite set of prime numbers and $A$ is a finite abelian group, we let $A_{(\pi)} =: \bigoplus_{p\in\pi} A_p$, where
$A_p$ is the $p$--primary subgroup of $A$. Our main result is 
\begin{thma}
Let $G$ denote a finite group acting freely and homologically trivially
on a rational homology 3--sphere $M$. Let $\pi $ denote the set
of precisely those primes which divide both $|G|$ and $|H_1(M;\bZ)|$.  
Then there exists an extension
$$1\to H_1(M;\bZ)_{(\pi)} \to Q_{\pi} \to G\to 1$$
where $ H_1(M;\bZ)_{(\pi)}$ is a central, cyclic subgroup and $Q_{\pi}$ is a finite
group with periodic cohomology of period two or four.
\end{thma}

%In the statement,   $H_1(M;\bZ)_{(\pi)}$ denotes the direct sum of the $p$-primary subgroups of  the finite abelian group $H_1(M;\bZ)$, for all primes $ p \in \pi$.

\smallskip
To investigate the corresponding existence  problem we raise the following question:

\begin{question} Let $G$ be a finite  group with periodic cohomology of period four. 
Does $G$ act freely and homologically trivially
on some rational homology 3--sphere~?
\end{question}

A complete list of such groups is given in Milnor \cite[\S 3]{milnor2}, and those which can act freely and orthogonally on $\bS^3$ were listed by Hopf  \cite{Hopf:1926}. Perelman \cite{Lott:2007} showed that the remaining groups in Milnor's list do not arise as the fundamental group of any closed, oriented $3$-manifold (some families were earlier eliminated by Ronnie Lee \cite{lee1}). For those with quaternion $2$-Sylow subgroups $Q(2^k)$, $k \geq 4$, we have a non-existence result in our setting.

 \begin{thmb} Let  $G$ be a finite group of period four which acts freely and homologically trivially on a rational homology $3$-sphere. Then  either  (i) $G$ is a quotient of a finite $3$-manifold group by a central cyclic subgroup, or (ii) the $2$-Sylow subgroup of $G$ is quaternion of order 8.
\end{thmb}
%\noindent\textbf{Remark.} 
\begin{Remark}
In Proposition \ref{prop:fourthree} we show that the quotient examples in part (i) include periodic groups with non-central elements of order two (which is impossible for groups acting freely on any sphere \cite[Corollary 1]{milnor2}). Proposition \ref{prop:four.7} establishes part (ii) by ruling out the groups with larger quaternion $2$-Sylow subgroups in Milnor's list. Moreover our result rules out all quotients of these groups by central cyclic subgroups.
\end{Remark}

 Among the groups in Milnor's list, the groups $Q(8n,k,l)$   with $2$-Sylow subgroup $Q(8)$ (i.e.~those with $n$ odd) have been much studied, and it is known that some (but not all) can act freely on integral homology $3$-spheres (see Madsen \cite{madsen2}). This work gives some new examples of existence in the setting of Theorem A via quotients by the action of central cyclic subgroups (see Proposition \ref{prop:typeA}).  The results of Pardon \cite{Pardon:1980} provided free actions of period four groups on rational homology $3$-spheres with some control on the torsion, but did not address the homological triviality requirement (see Proposition \ref{prop:pardon}). 
More information about the actions of the groups $Q(8n,k,l)$ is given in Theorem \ref{thm:existence}.

This article is organized as follows: in \S2 we apply methods from group cohomology to 
actions on rational homology spheres; in \S 3 we consider the restrictions arising in the homologically trivial case; in \S 4 we discuss the existence of homologically trivial actions,
and finally \S 5 deals with applications of our approach  to finite 
quotients of more general fundamental groups of closed 3--manifolds. 

\begin{acknowledgement}
We are grateful to the referees for their very helpful
comments.
\end{acknowledgement}

\section{Application of cohomological methods}
Let $G$ denote a finite group acting freely, smoothly, and preserving orientation
on a closed 3--manifold $M$
that is a rational homology sphere. In dimension three, free actions of finite groups by homeomorphisms
are equivalent to smooth actions, and the quotient manifolds are homotopy equivalent to finite CW complexes  (see \cite{Cairns:1961}).

We denote by $\Omega^r(\bZ )$ the $\bZ  G$ module uniquely defined in the stable category (where $\ZG$-modules are identified up to
stabilization  by projectives) as the $r$--fold dimension--shift of the trivial module $\bZ$.
Note the isomorphism of $\bZ G$--modules $H_1(M;\bZ)\cong H^2(M;\bZ)$;
we may use either version depending on the context. We refer to \cite{Adem:2004}
and \cite{Brown:1982}
for background on group cohomology,  and to \cite[Chap.~2, \S\S 5-6]{Carlson:2003} for a quick introduction to the stable category. 

\begin{proposition}
If a finite group $G$ acts freely 
on a rational homology 3--sphere $M$,
then there is a short exact sequence of $\bZ G$--modules in the
stable category of $\bZ G$--modules of the form 
\eqncount
\begin{equation}\label{eq:extclass}
0\to \Omega^{-2}(\bZ)\to \Omega^2(\bZ)\to H_1(M;\bZ)\to 0 
\end{equation}
\end{proposition}
\begin{proof} We will assume that $M$ is a $G$--CW complex with cellular
chains $C_*(M)$. Then we have exact sequences of $\bZ G$--modules

$$0\to \bZ\to C_3(M)\to C_2(M) \to B_1\to 0$$
$$0\to Z_1\to C_1(M)\to C_0(M)\to \bZ\to 0$$
$$0\to B_1\to Z_1\to H_1(M;\bZ)\to 0$$
where $B_1$ denotes the module of boundaries and $Z_1$ the module of cycles respectively.
The result follows from (stably) identifying $Z_1$ with $\Omega^2(\bZ)$ and $B_1$ with $\Omega^{-2}(\bZ)$
respectively.
\end{proof}

Note that the stable map $\Omega^{-2}(\bZ)\to \Omega^2(\bZ)$ defines an element
$$\sigma \in \underline{\rm Hom}_{\bZ G}(\Omega^{-2}(\bZ), \Omega^2(\bZ))\cong 
\underline{\rm Hom}_{\bZ G}(\bZ, \Omega^4(\bZ))\cong \wH^{-4}(G,\bZ)$$
This class appears when applying Tate cohomology to \eqref{eq:extclass}.

\begin{corollary}\label{long-exact-sequence}
The short exact sequence \eqref{eq:extclass} yields a long exact sequence in Tate cohomology
\[
\dots\to\wH^{i+2}(G,\bZ)\xrightarrow{\cup\,\sigma} \wH^{i-2}(G,\bZ)\to \wH^i(G,H_1(M;\bZ))\to \wH^{i+3}(G,\bZ)
\to\dots
\]
determined by the class $\sigma\in \wH^{-4}(G,\bZ)$ which is
the image of $1\in \wH^{0}(G,\bZ)\cong \cy{|G|}$.
\end{corollary}
Next we identify the class $\sigma$ geometrically.
\begin{proposition}\label{prop:fundclass}
If $[M/G]\in H_3(M/G,\bZ)$ denotes
the fundamental class of the quotient manifold, then $\sigma$ 
is the image of $c_*[M/G]\in H_3(BG,\bZ)$, under the natural isomorphism
$H_3(BG,\bZ)\cong \wH^{-4}(G,\bZ)$, where $c\colon M/G\to BG$ is the classifying map of the covering. 
\end{proposition}

\begin{proof}
The description due to MacLane \cite[Ch.~V.8]{maclane1} of 
$\Tor_3^{\bZ G}(\bZ, \bZ)\cong H_3(G;\bZ)$ via chain complexes, shows that the image
 of the fundamental class $c_*[M/G] \in H_3(G;\bZ)$ is represented by  the chain complex 
 $C_*(M)$ of finitely-generated free $\bZ G$-modules. We can apply dimension-shifting
  in the ``complete" $\Ext$-theory to the formula:
$$\Tor_3^{\bZ G}(\bZ, \bZ) = \wH^{-4}(G;\bZ) = \Ext^{-4}_{\bZ G}(\bZ;\bZ) = \underline{\Hom}_{\bZ G}(\bZ, \Omega^4\bZ)$$
to identify $c_*[M/G]$ with the extension class of the sequence \eqref{eq:extclass} 
(see Wall \cite[\S 2]{wall-ppr} for more background).
\end{proof}
Similarly the map $\Omega^2(\bZ)\to H_1(M,\bZ)$ defines an extension class 
$\cE_M \in
H^2(G, H_1(M;\bZ))$ which appears in the long exact sequence  from 
Corollary \ref{long-exact-sequence} as the image of the generator under the map
%\eqncount
%\begin{equation}
$\wH^{0}(G,\bZ)\to \wH^2(G,H_1(M;\bZ))$.
%\end{equation}
This algebraic map arises geometrically as folllows. Let $X\subset M/G$ denote a connected one dimensional 
$G-CW$
sub--complex such that $\pi_1(X)\to \pi_1(M/G)$ is onto. If we denote $F=\pi_1(X)$, then we have a diagram of extensions

\eqncount
\begin{equation} \vcenter{
\xymatrix{1 \ar[r]  & R \ar[r] \ar@{->>}[d] & F \ar[r] \ar@{->>}[d] & G \ar[r]\ar@{=}[d]& 1\\
1 \ar[r] & \pi_1(M) \ar[r] & \pi_1(M/G) \ar[r] & G \ar[r] & 1\\}}
\end{equation}
Abelianizing kernels gives rise to the diagram

\eqncount
\begin{equation}\label{diag:twofive}
 \vcenter{ \xymatrix{1 \ar[r]  & R_{ab} \ar[r] \ar@{->>}[d] & \Phi \ar[r] \ar@{->>}[d] & G \ar[r] \ar@{=}[d] & 1\\
1 \ar[r] & H_1(M;\bZ) \ar[r] & Q \ar[r] & G \ar[r] & 1\\}}
\end{equation}
where $\Phi$ is the associated free abelianized extension. This extension realizes the universal
class of highest exponent in $\wH^2(G, \Omega^2(\bZ))$; note that $R_{ab}$ is a free abelian
group which as a $\bZ G$--module is stably equivalent to $\Omega^2(\bZ)$. By construction,
the bottom extension
represents the class $\cE_M\in H^2(G, H_1(M,\bZ))$ (see \cite[p.~207]{Hilton:1997}).

It is known (see \cite{Browder:1983}) that free group actions can be fruitfully analyzed using exponents. For
actions on rational homology 3--spheres the analysis can be done quite explicitly. We recall

\begin{definition}
If $A$ is a finite abelian group, the exponent of $A$, $\rm{exp}(A)$ is defined as the smallest integer $N>0$ such that $Na=0$ 
for all $a\in A$. For an element $a\in A$, its exponent $\rm{exp}(a)$ is defined as the exponent of the subgroup generated by $a$.
\end{definition}

\begin{corollary}\label{exponents}
Let $G$ denote a finite group acting freely on a rational homology $3$--sphere $M$, then
$|G| = \rm{exp}\,(\sigma)\cdot\rm{exp}\,(\cE_M)$. The element $\sigma\in \wH^{-4}(G,\bZ)$
is invertible  in the graded ring $ \wH^{*}(G,\bZ)$ (and the group $G$ has periodic cohomology) 
if and only if the extension $Q$ representing $\cE_M$ is split. 
\end{corollary}
\begin{proof}
Consider the exact sequence in Corollary \ref{long-exact-sequence} at $i=2$:
$$ \dots \to \wH^4(G;\bZ) \xrightarrow{\cup\, \sigma} \wH^0(G;\bZ) \to \wH^2(G; H_1(M;\bZ)) \to \dots .$$
Recall that by Tate duality (see Brown \cite[Chap.VI.7]{Brown:1982}) there is 
an element 
$\sigma^*\in \wH^4(G,\bZ)$ such that 
$\sigma\cup \sigma^*= |G|/\rm{exp}\,(\sigma) \in \cy{|G|}$. This implies that the exponent of
the subgroup of $\wH^0(G;\bZ)$ generated by the image of $\cup\sigma$ is precisely $\rm{exp}(\sigma)$.
Furthermore from the definition of $\cE_M$, we see that
$\rm{exp}(coker\cup\sigma) = \rm{exp}(\cE_M)$. Using the fact that for cyclic groups the exponent and 
the order are equal we obtain the formula
$|G| = \rm{exp}\,(\sigma)\cdot\rm{exp}\,(\cE_M)$.

A finite group has periodic cohomology if and only if $\wH^*(G,\bZ)$ 
has an element of non--zero
degree with exponent equal to $|G|$ (see \cite[Chapter VI,Theorem 9.1]{Brown:1982}).
This element is invertible in Tate cohomology and cup product by this element is an isomorphism
in all degrees. From our formula we see that the 
class $\sigma$ has this highest exponent if and only if 
$\cE_M=0$, which is equivalent to the splitting of the extension $Q$.
\end{proof} 

\begin{remark}
If $\sigma$ is invertible in fact it can be shown that $\wH^i(H ,H_1(M,\bZ))=0$
for every subgroup $H\subset G$ and all $i\in \bZ$. In other words the 
module $H_1(M, \bZ)$ is cohomologically trivial. This holds for example if
$|G|$ is relatively prime to $|H_1(M;\bZ)|$.
\end{remark}

\begin{remark}
On the other extreme,
if $\sigma$ is trivial, then the extension class $\cE_M$ has highest exponent equal to $|G|$. Using the stable
isomorphism 
$$\underline{\rm Hom}_{\bZ G}(\Omega^{2}(\bZ), H_1(M,\bZ))\cong 
\underline{\rm Hom}_{\bZ G}(\bZ, \Omega^{-2}(H_1(M,\bZ)))$$ 
we can represent $\cE_M$ by a rank one trivial submodule in $\Omega^{-2}(H_1(M,\bZ))$ (note that any finitely
generated $\bZ G$--module is stably equivalent to a $\bZ$--torsion free module via 
dimension--shifting). 
By \cite[Theorem 1.1]{Adem:1989} and its proof,  the short exact sequence
$$0 \to \bZ \to \Omega^{-2}(H_1(M,\bZ)) \to \Omega^{-5}\bZ \to 0$$
in the stable category is split exact. 
After shifting back,  we obtain a stable decomposition
$H_1(M,\bZ)\cong \Omega^2(\bZ)\oplus \Omega^{-3}(\bZ)$.
This will occur for rational homology spheres with a free $G$--action
where $H^4(G,\bZ)=0$. An example of this phenomenon is given by
the Mathieu group $M_{23}$
(see \cite{Milgram:2000}).
\end{remark}

\begin{remark}\label{rem:twoeleven}
For $G=(\cy{p})^r$ we have that 
$p\cdot \wH^k(G,\bZ)=0$
for $k\ne 0$, by the K\"unneth Theorem, so we see that the exponent of $\cE_M$ must be at least $p^{r-1}$.  
In particular the module $H_1(M;\bZ)$
must have $p^{r-1}$--torsion. 
\end{remark}

\begin{remark}\label{differentials}
Given a free action of a finite group $G$ on a rational homology sphere $M$,
we can consider the
Serre spectral sequence for the fibration $M\to M/G\to BG$. For $i\ge 1$, the maps $\wH^{i}(G,\bZ) \to \wH^{i+2}(G, H_1(M,\bZ))$ in the long exact sequence of Corollary \ref{long-exact-sequence} can be identified with the
differentials $d_2\colon E^{i,3}_2\to E^{i+2,2}_2$. Similarly the maps $\wH^i(G, H_1(M,\bZ))\to H^{i+3}(G,\bZ)$
can be identifed with the differentials $d_3\colon E^{i,2}_3\to E^{i+3,0}_3$. In particular 
exactness at the term $\wH^r(G, H^2(M,\bZ))$ reflects 
the fact that
$E^{r,2}_{\infty}=0$ for $r\ge 2$ given that $M/G$ is a 3--manifold. 
\end{remark}

\section{Restrictions in the homologically trivial case}
In this section we focus on the special case when the $G$--action on $M$ is trivial in homology. This
imposes some drastic restrictions.

\begin{proposition}
If $G$ acts freely and homologically trivially on a rational
homology 3--sphere $M$, then every elementary abelian subgroup
of $G$ has rank at most two.
\end{proposition}
\begin{proof}
 We may assume that $G = (\cy p)^r$ for some prime $p$. Note that $\wH^k(G,H_1(M;\bZ))$ has exponent $p$
when $H_1(M;\bZ)$ has trivial action and $k\ne 0$, by the K\"unneth Theorem.
From Corollary \ref{exponents}  we see that if $(\cy{p})^r$
acts freely and homologically trivially on $M$, then $p^r$ divides $p^2$ and the result follows.
\end{proof}

Let us write the trivial $\bZ G$--module $H_1(M;\bZ)$ as a direct sum
of finitely generated, finite abelian $p$--groups $A_p = H_1(M;\bZ)_{(p)}$. 
Then we have

\begin{lemma}\label{cyclic}
If $p$ is a prime number dividing the order of $G$, then $H_1(M;\bZ)_{(p)}$ is either trivial or cyclic.
\end{lemma}
\begin{proof}
Consider a cyclic $C\cong \cy{p}$ in $G$. Since $H^7(C;\bZ) = 0$, from the
 sequence in Corollary \ref{long-exact-sequence} for $C$ at $i=4$, we see that
$H^4(C, H_1(M;\bZ))$ is a homomorphic image of $H^2(C,\bZ)= \cy p$. It follows  that 
$H^4(C,H_1(M;\bZ))\cong H^4(C, A_p) \cong  A_p/pA_p$ 
is either trivial or $\bZ /p$, which
proves the result. 
\end{proof}

%A well--known example of a free action on a rational homology sphere is given by
%the free action of $\bZ /2 \times \bZ /2$ on $\RP^3$, which comes from the free action
%of the quaternions on the 3--sphere. Associated to it there is a central extension of the form
%$$1\to \cy{2} \to Q(8) \to \cy{2} \times\cy{2}\to 1. $$ 
 Next we will show that all groups acting freely
and homologically trivially on rational homology spheres can be described as quotients of periodic groups.

\begin{theorem}\label{thm:threethree}
Let $G$ denote a finite group acting freely and homologically trivially
on a rational homology 3--sphere $M$. Let $\pi = p_1\dots p_r$ denote the product
of precisely those primes which divide both $|G|$ and $|H_1(M;\bZ)|$.  
Then there exists an extension
$$1\to H_1(M;\bZ)_{(\pi)} \to Q_{\pi} \to G\to 1$$
where $H_1(M;\bZ)_{(\pi)}$ is a central, cyclic subgroup and $Q_{\pi}$ is a group with periodic cohomology.
\end{theorem}

\begin{proof}
Suppose that $G$ acts freely on a rational homology 3-sphere $M$ and consider the group extension
$$1\to \pi_1(M)\to \pi_1(M/G)\to G\to 1.$$
Let $L$ denote the kernel of the map $\pi_1(M)\to H_1(M;\bZ)_{(\pi)}$;
then it is normal in both $\pi_1(M)$ and $\pi_1(M/G)$ and
we can consider the associated central quotient extension:
$$0\to H_1(M;\bZ)_{(\pi)}\to Q_{\pi} \to G\to 1.$$ 
 Note that $H^2(G,H_1(M;\bZ))\cong H^2(G,H_1(M;\bZ)_{(\pi)})$, since $H^2(G; A_p)= 0$ for $p \nmid |G|$, where $A_p$ denotes the $p$-primary part of $H_1(M;\bZ)$ as before.
Our construction of the extension for $Q_\pi$ is the obvious quotient of the extension for $Q$ representing the class
$\cE_M$ apprearing in Corollary \ref{exponents}.

Suppose $p$ is a prime that divides $|G|$ but which is relatively prime to $|H_1(M;\bZ)|$.
Then $H^2(\Syl_p(G), H_1(M;\bZ))=0$ and so $\Syl_p(G)= \Syl_p(Q_{\pi})$ is periodic by Corollary \ref{exponents}. Now 
suppose that $p$ is a prime which divides $\pi$, and let $C\subset G$ denote a cyclic 
subgroup of order $p$. By naturality we have a commutative diagram, where
the rows are exact sequences:
\[  
\xymatrix{{\wH^0(G, \bZ)}  \ar[r] \ar[d] & {\wH^2(G, H_1(M;\bZ))}\ar[r] \ar[d] & {\wH^5(G,\bZ)} \ar[d] \\
{\wH^0(C, \bZ)} \ar[r]^(0.4){\approx} & {\wH^2(C, H_1(M;\bZ))} \ar[r] & 0\\}
\]
The isomorphism in the lower row of this diagram comes from the rest of the sequence
$$0 \to \wH^1(C; H_1(M;\bZ)) \to \wH^4(C;\bZ) \to \wH^0(C;\bZ) \to \wH^2(C, H_1(M;\bZ)) \to 0$$
since $\wH^1(C; H_1(M;\bZ))  \cong \cy p \cong \wH^4(C;\bZ)$.

By Lemma \ref{cyclic}  the $p$--component of $H_1(M;\bZ)$ is a finite cyclic $p$--group with a trivial
$C$--action. Hence $\wH^i(C, H_1(M;\bZ))\ne 0$ for all $i$. The map
$\wH^0(G,\bZ)\to \wH^0(C,\bZ)$ sends a generator to a generator so the extension class
$\cE_M \in H^2(G, H_1(M;\bZ))$ restricts non--trivially on all such subgroups $C$, and hence the
corresponding restricted extensions of the form
$0\to H_1(M;\bZ)_{(\pi)}\to Q_{\pi} |_C \to C\to 1$
are all non--split.  

If we take $H=H_1(M;\bZ)_{(\pi)}$, which we know to be 
cyclic by Lemma \ref{cyclic}, then the extension expresses $G$ as the quotient $Q_{\pi}/H$ where $H$ is a central,
cyclic subgroup and every restricted group of the form $Q_{\pi}|C$ is non--split, where
$C\cong \cy{p} $, and $p$ divides $\pi$. 
Let $u\in Q_{\pi}$ denote an element of order $p$; if the subgroup generated by 
$H$ and $u$ is not
cyclic, then it must be split abelian, a contradiction. Therefore all elements of order
$p$ in $Q_{\pi}$ lie in $H$, a cyclic subgroup, and so $Q_{\pi}$ has no rank two $p$--elementary abelian subgroups. 
We have already established this for the primes which do not divide $\pi$, whence we 
infer that $Q_{\pi}$ has periodic cohomology. 
\end{proof}

\begin{proposition}\label{prop:threefour}
The period of $Q_\pi$  is  two or four.
\end{proposition}

\begin{proof} 
Consider the central group extension
$1\to H \to Q \to G\to 1$ where $H:=H_1(M,\bZ)_{(\pi)}$ and $Q: =Q_\pi$. 
By Swan \cite{Swan:1960}, the period of $Q$ is the least common multiple of the $p$-periods of $Q$ taken over all primes $p$ dividing $|Q|$. The $p$-periods are determined by group cohomology with $p$-local coefficients. By \cite[Theorem 1]{Swan:1960}, the $2$-period of $Q$ is 2 or 4.   Moveover, by \cite[Theorem 2]{Swan:1960}, the $p$-period of $Q$ for $p$ odd is twice the order of $\Phi_p(Q)\cong N_Q(\Syl_p(Q))/C_Q(\Syl_p(Q))$, the group
of automorphisms of $\Syl_p(Q)$ induced by inner automorphisms of $Q$. 
Note that as $\Syl_p(Q)$ is cyclic, its automorphism group is also cyclic and hence $\Phi_p(Q)$ is cyclic of order prime to $p$.

As explained in \cite[Lemma 3]{Swan:1960}, the action on $\wH^{2i}(\Syl_p(Q),\bZp)$ is given by multiplication by $r^i$, where $r$ is an
integer prime to $p$ that is a multiplicative generator of $\Phi_p(Q)$. Hence this action has invariants only when
$i$ is a multiple of $|\Phi_p(Q)|$, and $\wH^*(Q,\bZp)\ne 0$ only in  degrees which are multiples
of $2|\Phi_p(Q)|$.

 If $p$ divides $|G|$ then
the projection $Q\to G$ induces an isomorphism $\Phi_p(Q) \cong \Phi_p(G)$, and hence the $p$-periods of $Q$ and $G$ are equal. 
Consider now the following portion of the $p$--local version of the long exact cohomology sequence from Corollary 2.3: 
 %\eqncount
 %\begin{equation}\label{eqn:plocal}
  $\wH^4(G;\bZp) \to \wH^0(G;\bZp) \to \wH^2(G;H_p)$,
  %\end{equation}
where $H_p=\Syl_p(H) = H_1(M;\bZp)$.
 As $|G|$ is divisible by $p$, the middle term is non-zero.
Now if $\wH^2(G;H_p) = 0$, then $\wH^4(G;\bZp)\ne 0$
and we conclude 
that $G$  has $p$-period dividing four (a $p$--local version of Corollary \ref{exponents}).  
However, by the universal coefficient 
theorem applied to the trivial $G$--module $H_p$, we see  that
$\wH^2(G,H_p)=0$ if the $p$--period of $G$ is four or higher. 
Hence we conclude that 
the $p$-periods of $G$ and  $Q$ must both be either two or four.  %Let
\end{proof}
Theorem A follows from Theorem 3.3 and Proposition 3.4.
The structure of $G$ is more explicit for $p$-groups. 
\begin{corollary}
A finite $p$--group $G$ acts freely and homologically trivially on some rational homology 3--sphere $M$
with non--trivial $p$--torsion in $H_1(M;\bZ)$ if and only if (i) $G$ is cyclic  or 
(ii) $p=2$, $H_1(M; \bZ)_{(2)}\cong \cy{2}$ and
$G$ is a dihedral group.
\end{corollary}
\begin{proof} The finite groups of the form $Q/H$ where $Q$ is a periodic $p$--group and $H$ is a 
non--trivial central cyclic subgroup are precisely the cyclic groups and $Q(2^n)/Z(Q(2^n))$, where
$Q(2^n)$ is a generalized quaternion group of order $2^n$, $n\ge 3$, with centre isomorphic to
$\cy{2}$ and quotient a dihedral group of order $2^{n-1}$. Conversely all the groups $Q$ appearing
above act freely, on $\mathbb \bS^3$, hence all the quotients $G=Q/H$ act 
freely on a rational homology sphere; note these actions are homologically trivial as $H_1(\mathbb \bS^3/H, \bZ)
\cong H$, a central subgroup.
\end{proof}

\begin{corollary}
Let $G$ act freely and homologically trivially
on a rational homology 3--sphere $M$.
\begin{enumerate}
\item
If both $|G|$
and $|H_1(M;\bZ)|$ are even, then $\Syl_2(G)$  is either cyclic or dihedral.
\item
If $p$ is an odd prime dividing $|G|$, then  $\Syl_p(G)$ is cyclic. 
\item 
If $(|H_1(M; \bZ)|, p)=1$ then $\Syl_p(G)$ is either cyclic or generalized quaternion.
\end{enumerate}
\end{corollary}

\begin{remark}
It is also an interesting problem to determine which
 groups can act homologically trivially on higher dimensional rational homology spheres.
  Using exponents it can be shown that 
if $G$ acts freely and homologically trivially on a simply--connected rational homology
$n$--sphere, then the rank of $G$ can be at most $n-2$. 
We expect that further group theoretic restrictions 
will play a role.
\end{remark}

\section{Existence of homologically trivial actions}

As mentioned in the Introduction, any finite group can act freely on some rational homology 
$3$-sphere if there is no homological triviality assumption. 
This was first proved by Pardon \cite{Pardon:1980} using local surgery theory, extending a result of Browder and Hsiang 
\cite[p.~267]{Browder:1978}. The direct 3-dimensional argument of Cooper 
and Long \cite{Cooper:2000} avoids the surgery formalism, but does not give any control on the torsion in $H_1(M;\bZ)$.

\begin{proposition}\label{prop:pardon}
Let $G$ be a finite group, and $p$ a prime such that $(p, |G|) = 1$. Then $G$ acts freely on some $p$-local homology $3$-sphere.
\end{proposition}

\begin{proof}This statement is a special case of Pardon \cite[Theorem B]{Pardon:1980}, 
together with the standard remark that high-dimensional surgery existence results in dimensions $4k+3 \geq 7$, imply existence results in dimension three up to homology.  
\end{proof}

It appears to be much more difficult to solve the existence problem for a \emph{given} rational homology $3$-sphere $M$. For example, what if we consider only the space form groups but do not require homologically trivial actions ?
\begin{question} 
 If $G$ acts freely on $\bS^3$, can it act freely on a given rational homology $3$-sphere
$M$ with $(|G|, |H_1(M;\bZ)|) = 1$~?
\end{question}

We will now use the information in Theorem A to make some remarks about the existence of \emph{homologically trivial} actions on rational homology
$3$-spheres. 

The finite groups which can act freely on $\bS^3$ are now known (by the work of Perelman \cite{Lott:2007}): they are precisely 
the periodic groups in Hopf's list \cite{Hopf:1926}. For any of these groups we can obtain examples of homologically trival actions on quotients $\bS^3/H$, where $H$ is a central cyclic subgroup. 
 \begin{example}\label{ex:dihedral}
Let $H$ denote the quotient of the binary dihedral group $Q(4n)$ by its unique central subgroup of order two. Then $H$ is dihedral of order $2n$, and has rank two if $n$ is even. 
 \end{example}

Of the remaining period four groups  (those which do not act freely on $\bS^3$), we first consider those which do not satisfy Milnor's $2p$-condition,  which requires that every subgroup of order $2p$, for $p$ prime,  must be cyclic (see \cite[Corollary 1]{milnor2}).

\begin{proposition}\label{prop:fourthree}
 Let $G$ be a finite group with periodic cohomology of period four containing a non-cyclic subgroup of order $2p$, for some odd prime $p$. Then $G$ is  the  product of a dihedral group by a cyclic group of relatively prime order. Any quotient of $G$ by a central cyclic subgroup can act freely and homologically trivially on some rational homology $3$-sphere.
\end{proposition}
\begin{proof}  This follows by checking the list of periodic groups, taking into account our period four assumption. A convenient reference is Wall \cite[Theorem 4.5]{Wall:2013}, which states that $G$ is an extension of a normal subgroup $G_0$ of odd order by a group $G_1$ isomorphic to one of the form $C(2^k)$, $Q(2^k)$, $T_v^*$, $O_v^*$, $SL_2(p)$ or $TL_2(p)$. The periods of these groups are listed in \cite[Corollary 5.6]{Wall:2013}. In our case, $G_0$ must be cyclic (the only odd order group with period $\leq 4$), and $G_1 = C(2^k)$ since there is an unique element of order two in the other cases. The action of $G_1$ on $G_0$ must be faithful to violate the $2p$ condition, and $G$ of period four implies the claimed structure for $G$. 

Since any quotient of $G$ by a central subgroup is again of the same form,  the required actions arise by quotients of a free action on $\bS^3$ by products of a binary dihedral group $Q(4n)$, for $n$ odd, with a cyclic group of coprime order.
\end{proof}

Milnor \cite[\S 3]{milnor2} listed the period four groups which do satisfy the $2p$ conditions, and identified two families of such group which (by Perelman \cite{Lott:2007})
can not act freely on $\bS^3$. These are:
\begin{enumerate}
\item $Q(8n, k, l)$, with  $n>k>l \geq 1$, and $8n$, $k$, $l$ pairwise relatively prime;
\item $O(48, k,l)$, with $l$ odd, $3 \nmid l$ and 48, $k$, $l$ pairwise relatively prime.
\end{enumerate}

One can also take the product of any one of these groups with a cyclic group $C(r)$ of relatively prime order. We will refer  to those listed in (i) as \emph{type A} if $n$ is odd, or \emph{type B} if $n\geq 2$ is even,  and  those in (ii) as  \emph{type C} for $O(48,k,l)$. The groups of type B or C have orders divisible by 16. The direct product of one of these groups with a cyclic group of coprime order will be called a \emph{generalized} type A, B or C periodic group.
From the presentations given in \cite[\S 3]{milnor2}, one can check that the groups of type A, B, or C all have quaternion $2$-Sylow subgroup, and contain no non-trivial central subgroups of odd order. 

\medskip
The following is a useful observation.
\begin{lemma}\label{lem:odd} Let $G$ be a period four group with quaternion $2$-Sylow subgroup. If $G$ acts freely and homologically trivially on a rational homology $3$-sphere $M$,  then $H_1(M;\bZ)$ has odd order.
\end{lemma}
\begin{proof}
From diagram \eqref{diag:twofive}, we have a short exact sequence:
$$1 \to H_1(M;\bZ) \to Q\to G \to 1. $$
Note that $S=Syl_2(G)$ is a quaternion group. Now the restricted sequence 
from Corollary \ref{long-exact-sequence} with $i=2$ gives an injection 
$$ 0 \to H^1(S;H_1(M;\bZ)) \to \wH^4(S;\bZ) \to \wH^0(S;\bZ).$$
Since $S$ has period four, $\wH^4(S;\bZ) = \cy{|S|}$, whereas the group $H^1(S;H_1(M;\bZ))$ contains 
$\cy{2} \times \cy{2}$ unless $H_1(M;\bZ)$ is of odd order.
\end{proof}

We have a sharper result for the groups in Milnor's list.
\begin{proposition}\label{prop:four.5}
Suppose that $G$ is a period four group of type \textup{A}, \textup{B} or \textup{C}. If $G$ acts freely and homologically trivially on a rational homology $3$-sphere $M$,  then we have  $(|H_1(M;\bZ)|, |G|) =1$.
\end{proposition}
\begin{proof} 
 We start again from the short exact sequence in diagram \eqref{diag:twofive}:
$$1 \to H_1(M;\bZ) \to Q\to G \to 1. $$
Consider the pushout sequence from Theorem \ref{thm:threethree}:
$$1 \to H_1(M;\bZ)_{(\pi)} \to Q_\pi \to G \to 1,$$
and note that $Q_\pi$ is a period four group by Proposition \ref{prop:threefour}.

Since $2 \notin \pi$, it follows that $Q_\pi$ is  an extension of $G$ by a central cyclic subgroup of odd order.   Since $G$ has type A, B or C,  the group $G$ contains no non-trivial central subgroups of odd order. It follows that 
 $H_1(M;\bZ)_{(\pi)}=0$, and 
hence $Q_\pi = G$.
In particular, this implies that $H_1(M;\bZ)$ has order relatively prime to $|G|$.
\end{proof}

We need one more observation about the structure of the groups in Milnor's list.
\begin{lemma}\label{lem:quotient}
Let $G$ be a non-periodic quotient of a generalized type \textup{B} or \textup{C} period four group by a central cyclic subgroup. Then $G = C(k) \times G_1$, where $G_1$ is the 
quotient of a type \textup{B} or \textup{C} period four group by a central cyclic subgroup, and $k$ is coprime to the order of $G_1$.
\end{lemma}
\begin{proof}
Suppose that $G$ is a non-periodic quotient of a generalized type B or C group by a central cyclic subgroup  $C(s)$. We have an exact sequence of the form
 $$1 \to C(s)\to C(r)  \times Q_1 \to G  \to 1$$
 where $Q_1$ has type B or C, and $r$ has (odd) order coprime to $|Q_1|$. Since $Q_1$ has a unique (central) element of order two and $G$ is non-periodic, $s = 2s_1$, where $s_1$ is odd.  Since $Q_1$ has no non-trivial central subgroups of odd order, we have $r = ks_1$ and $G = C(k) \times G_1$, where $G_1$ is a non-periodic quotient of $Q_1$. 
\end{proof}
These results allow us to rule out the groups of (generalized) types B and C. Note that by Lemma \ref{lem:odd},  if such a group  acts freely and homologically trivially on a rational homology $3$-sphere $M$, then $M$ must be a $\bZ_{(2)}$-homology sphere.
Theorem B follows from the following result.

\begin{proposition}\label{prop:four.7}
 Let $G$ be the quotient of a (generalized) type \textup{B} or \textup{C} period four group by a central cyclic subgroup. Then 
$G$ can not act freely and homologically trivially on a rational homology $3$-sphere.
\end{proposition}
\begin{proof} 
The period four groups of (generalized) type B or C themselves can not act freely on any
 $\bZ_{(2)}$-homology $3$-sphere (see Ronnie Lee \cite[Corollary 4.15, Corollary 4.17]{lee1}), so they are ruled out by  Lemma \ref{lem:odd}.

 By Lemma  \ref{lem:quotient}, and restricting the action if necessary, we may assume that $G$ is the quotient of a type B or C group by a central cyclic subgroup.

 Now suppose that  $G$  acts freely and homologically trivially on a rational homology $3$-sphere $M$. 
   We then have a covering space $M \to M/G \to BG$.
 From \eqref{diag:twofive}, we have an
 extension $1 \to H_1(M;\bZ) \to Q \to G \to 1$. As above, we consider
 the pushout $1 \to H_1(M;\bZ)_{(\pi)}\to Q_\pi \to G \to 1$,
 where $Q_\pi$ is a period four group and $H_1(M;\bZ)_{(\pi)}$ is a central cyclic subgroup of $Q_\pi$.

  Since $G$ is non-periodic, $2 \in \pi$ and
 $H_1(M;\bZ)$ has order $2d$, where $d$ is odd.  
 It follows from the explicit presentations in \cite[\S 3]{milnor2}  that $Q_\pi$ must again be of type B or C, and $H_1(M;\bZ)_{(\pi)}$ must contain the unique central subgroup 
 $T = C(2)$ of order two in $Q_\pi$.

 The group $Q_\pi$ is constructed by a pushout from $\pi_1(M)$, and we can  form a further pushout over the projection $H_1(M;\bZ) \to T= C(2)$ to obtain the group extension
 $$1 \to T \to Q' \to G \to 1$$
 in which $Q'$ is again a period four group of type B or C. The $2$-fold covering $M' \to M$ given by the quotient $\pi_1(M) \to T$, followed by the $G$-covering 
 $M \to M/G$, is just the $Q'$-covering $M' \to M/G$.

 To obtain a contradiction, we will now show that $M'$ is a $\bZ_{(2)}$-homology sphere. 
 From the structure of $M' \to M$ as a $2$-fold covering, we have an exact sequence
$$ 0 \to H_0(C(2);H_1(M';\bZ)) \to H_1(M;\bZ) \to \cy 2 \to 0$$
and $|H_1(M;\bZ)| = 2d$, with $d$ odd. Hence the co-invariants 
$H_0(C(2);H_1(M';\bZ))$ are of odd order, and 
 $H_1(M';\bZ)$ has no $2$-torsion. We have an  exact sequence of $C(2)$-modules involving the torsion subgroup and the torsion-free quotient:
 $$0 \to H_1(M';\bZ)_{tors} \to H_1(M';\bZ) \to \bZ^r \to 0$$
 and by applying group homology $H_*(C(2); -)$ to the sequence, we see that $H_0(C(2); \bZ^r)$ is a direct sum of $r$ cyclic groups (each one either $\bZ$ or $\cy 2$,  depending on whether the module action of $C(2)$ on  each summand of $\bZ^r$ is trivial or non-trivial), and conclude that $r=0$. Hence
 $H_1(M';\bZ)$ is all odd torsion. In other words, $M'$ is a $\bZ_{(2)}$-homology $3$-sphere and the free $Q'$-action can not exist.
\end{proof}

The remaining existence question concerns central quotients of the period four groups of type A.  It is enough to consider the period four groups themselves.

\begin{proposition}\label{prop:typeA}
Let $G$ be the quotient of a type \textup{A} period four group $Q$ by a central cyclic subgroup $T\leq Q$. 
If $Q$ acts freely and homologically trivially
on a rational homology $3$-sphere $M$,  then $G$ acts freely and homologically trivially on
 $M/T$, which is again a rational homology $3$-sphere. 
 \end{proposition}

\begin{proof}
Let $G$ be the quotient of a type \textup{A} period four group $Q$ by a central cyclic subgroup $T\leq Q$.
If $Q$ acts freely and homologically trivially
on a rational homology $3$-sphere $M$,  then $G$ acts freely on
the rational homology $3$-sphere $M/T$. 

It remains to show that the $G$-action on $M/T$ is homologically trivial. Since $T$ is central, the covering space $M \to M/T$ is preserved by $Q$, and we have an exact sequence:
$$ 0 \to H_1(M;\bZ) \to H_1(M/T;\bZ) \to H_1(T;\bZ) \to 0$$
since $H_2(T;\bZ) = 0$ and $T$ acts homologically trivially on $M$. Looking at the associated long exact
sequence in group cohomology and noting that $Q$ acts trivially on both $H_1(M; \bZ)$ and $H_1(T; \bZ)$, 
we obtain:
$$ 0 \to H_1(M;\bZ) \to H_1(M/T;\bZ)^Q \to H_1(T;\bZ) \to H^1(Q; H_1(M;\bZ))\ .
$$
However, by Proposition \ref{prop:four.5} we have
$(|H_1(M;\bZ) |, |Q|) = 1$, and hence $H^1(Q; H_1(M;\bZ)) =0$, Therefore $H_1(M/T;\bZ)^Q = H_1(M/T;\bZ)$, and the $G$-action on $M/T$ is homologically trivial.
\end{proof}

The period four groups $G = Q(8n,k,l)$ of type A can not act freely on $\bS^3$ (by Perelman), but some members of this family do act freely on integral homology $3$-spheres. For the existence of such actions, there are two obstacles: a finiteness obstruction and a surgery obstruction.
Swan \cite{swan1} showed that  for every period four group, there exists a finitely dominated Poincar\'e $3$-complex $X$ with $\pi_1(X) = G$  and universal covering $\widetilde X \simeq \bS^3$. Such  a complex is called a \emph{Swan complex} of type $(G,3)$. 

We recall that the homotopy types of $(G,3)$-complexes are in bijection (via the first $k$-invariant) with the invertible elements in $\widehat H^4(G;\bZ) \cong \cy{|G|}$. 

\begin{lemma}\label{lem:four.9} Let $G$ be a period four group which acts freely and homologically trivially on a rational homology $3$-sphere $M$. Then there exists a $(G,3)$-complex $X$, unique up to homotopy, and a degree 1 map $f\colon M/G \to X$ compatible with the classifying maps of the  $G$-fold coverings. 
\end{lemma}
\begin{proof}
The classifying map $c\colon M/G \to BG$ of the covering $M \to M/G$ gives a class $c_*[M] \in H_3(G,\bZ)$. By Proposition \ref{prop:fundclass}, this class corresponds to a generator 
$$\sigma^* \in \wH^4(G;\bZ) \cong \wH^{-4}(G;\bZ) \cong H_3(G;\bZ). $$
  Let $X$ be the $(G,3)$-complex defined (up to homotopy) by this $k$-invariant. Since the classifying map $c\colon M/G \to BG$ is surjective on fundamental groups, it follows that $c$ lifts to a map $f\colon M/G \to X$. Since the images of the fundamental classes of $M/G$ and $X$ agree in $H_3(G;\bZ)$, it follows that $f$ has degree 1. 
\end{proof}

\begin{remark}
Any degree 1 map $f\colon N \to X$ from a closed oriented $3$-manifold to a $(G,3)$-complex provides a degree 1 normal map by pulling-back a framing of the trivial bundle over $X$.
\end{remark}\
The Wall finiteness obstruction $\sigma(X) \in  \widetilde K_0(\bZ G)$ vanishes if and only if there exists a finite Swan complex of type $(G,3)$. This is the first obstruction to existence.
By varying the homotopy type of $X$, Swan defined
an invariant 
$\sigma(G) \in \widetilde K_0(\bZ G) / S(G)$,
depending only on $G$, 
where $S(G) \subseteq \widetilde K_0(\bZ G) $ is the \emph{Swan subgroup} generated by projective ideals of the form $\la r, \Sigma\ra \subset \bZ G$, where $(r, |G|) = 1$ and $\Sigma$ denotes the norm element. Then $\sigma(G) = 0$ if and only if $\sigma(X) \in S(G)$ for every Swan complex $X$ of type $(G,3)$.

\begin{proposition}\label{prop:finiteness}
 Let $G = Q(8n, k, l)$, with $n$ odd, be a period four group of type \textup{A}. If $G$ acts freely and homologically trivially on a rational homology $3$-sphere, then $\sigma(G) =0$.
\end{proposition} 
\begin{proof} Under the given assumptions,  $G$ acts freely and homologically trivially on a rational homology $3$-sphere $M$, such that
 $(|H_1(M;\bZ)|, |G|) = 1$. By Lemma \ref{lem:four.9}. the classifying map $M/G \to BG$ of the covering $M \to M/G$ lifts to a degree 1 map $f\colon M/G \to X$, to a uniquely defined  $(G,3)$-complex $X$. Since the map $f$ induces a surjection on fundamental groups, the argument of Mislin \cite[Theorem 3.3]{Mislin:1976} shows that $\sigma(M/G) = \sigma(X) + \la d, \Sigma\ra \in \widetilde K_0(\ZG)$, where $d = |H_1(M;\bZ)|$,
 and hence $\sigma(X) \in S(G)$. Since varying the homotopy type of $X$ changes $\sigma(X)$ only by an element of the Swan subgroup (see Swan \cite[Lemma 7.3]{swan1}), we conclude that $\sigma(G) = 0$.
\end{proof}

The secondary obstruction comes from surgery theory (and is defined only if the finiteness obstruction is zero). It can be computed in some cases to show  existence (see Madsen \cite{madsen2}). 
 For the type A groups,  
 a  $(G,3)$-complex $X$ has \emph{almost linear} $k$-invariant $e_0 \in H^4(G;\bZ)$ if the restriction of $e_0$ to each Sylow subgroup of $G$ is the $k$-invariant of a standard free orthogonal action on $\bS^3$ (see \cite[p.~195]{madsen2}). 
 
 \begin{definition} We will say that a free homologically trivial action of  a type A group $G$ on a rational homology $3$-sphere $M$ has \emph{almost linear $k$-invariant} if there exists a degree 1 map $f\colon M/G \to X$ to a  \emph{finite} $(G,3)$-complex with almost linear $k$-invariant $e_0 \in H^4(G;\bZ)$.
\end{definition}
 \begin{remark}\label{rem:ZHS}
 If $G$ acts freely and smoothly on an integral homology $3$-sphere $\Sigma$, then the quotient manifold $\Sigma/G = X$ is a finite $(G,3)$-complex with almost linear $k$-invariant (see \cite[Corollary C]{hmadsen3} and the discussion of \cite[Conjecture D]{madsen2}). By Proposition \ref{prop:typeA}, any quotient of such a group $G$ by a central cyclic subgroup would act freely and homologically trivially on a rational homology $3$-sphere.
 \end{remark}

Conversely, we expect that the following existence statement holds:
\begin{conjecture} Let $G = Q(8n,k,l)$, with $n$ odd,  be a period four group of type \textup{A}. Then $G$ acts freely and homologically trivially on a rational homology $3$-sphere with almost linear $k$-invariant if and only if $G$ acts freely on an integral homology $3$-sphere.
\end{conjecture} 
   
   In the remainder of this section, we prove this conjecture under some additional assumptions. 
 If $f\colon N \to X$ denotes a degree 1 normal map to a finite $(G,3)$-complex, with $\pi_1(X) = G$, then there is \emph{weakly simple} surgery obstruction $\lambda'(f) \in L'_3(\ZG)$. This is defined since every finite Poincar\'e $3$-complex with finite fundamental group is weakly simple (meaning that its Poincar\'e torsion lies in $SK_1(\ZG)$). We let $\lambda^h(f) \in L_3^h(\ZG)$, the image of $\lambda'(f)$ under the natural map, denote the obstruction to surgery on $f$ up to homotopy equivalence.

Let $H = Q(4ab)$ denote the index two subgroup of $G$, containing the subgroup $C(4) \leq Q(8)$ which acts by inversion on the normal subgroup of order $ab$.

\begin{theorem}[{Madsen \cite{madsen2}}]\label{thm:res}
 Suppose that $G = Q(8n, k, l)$, with $n$ odd, is a period four group of type \textup{A} such that $\sigma(G) = 0$. Let $f\colon N \to X$ be a degree 1 normal map to a finite  $(G,3)$-complex with almost linear $k$-invariant. Then $\lambda'(f) = 0 $ if and only if $\Res_H(\lambda'(f)) =0$ for each subgroup $H\leq G$ of the form $H = Q(4ab)$. Furthermore, $\lambda^h(f) = 0 $ if and only if $\Res_K(\lambda^h(f)) =0$ for each subgroup $K \leq G$ of the form $ K = Q(8a,b)$.
\end{theorem}
\begin{proof} 
This a a summary statement of the calculations in \cite[\S 4-5]{madsen2}. See in particular \cite[Theorems 4.19, 4.21 and Corollary 5.12]{madsen2}.
\end{proof}

For the groups $G = Q(8a,b)=Q(8a,b,1)$, the \emph{top component} $S(ab) \subseteq S(G) \subset \widetilde K_0(\ZG)$  of the Swan subgroup is defined as the the kernel of the restrictions to all odd index subgroups. For example, Bentzen and Madsen \cite[Proposition 4.6]{Bentzen:1983} computed $S(Q(8p,q))$, for $p,q$ odd primes, almost completely, and showed that $S(pq) = 0$ in a many cases (e.g $(p,q) \equiv (\pm 3, \pm 3) \mod 8$;  or $(p,q) \equiv (1,\pm 3) \mod 8$, and $2$ has odd order mod $p$.

\begin{theorem}\label{thm:existence}
Let $G = Q(8p,q)$, for odd primes $p>q$, and assume that $S(pq) = 0$.
Then $G$ acts freely and homologically trivially on a rational homology $3$-sphere with almost linear $k$-invariant if and only if $G$ acts freely on an integral homology $3$-sphere.
  \end{theorem}
  \begin{proof} Remark \ref{rem:ZHS} explains the sufficiency part. For the converse,   suppose that $G = Q(8p,q)$ with $S(pq) = 0$ acts freely and homologically trivially on a rational homology $3$-sphere $M$ with almost linear $k$-invariant.
 Then there exists a finite $(G,3)$-complex $X$ with almost linear $k$-invariant, and a degree 1 normal map $f\colon M/G \to X = X(G)$. 
  by \cite[Theorem 3.1]{h4}, we may assume that the covering space $X(H)$ is homotopy equivalent to an orthogonal spherical space form, for $H = Q(4pq) \leq G$, and that the normal invariant restricts to the normal invariant of an orthogonal spherical space form over the $2$-Sylow covering $X(Q(8))$.
 In particular, since $|H_1(M;\bZ)| = d$ is odd, we must have that
 $\sigma(X(Q(8)) = \la d, \Sigma\ra = 0 \in \widetilde K_0(\bZ Q(8))$.
  Hence $d \equiv \pm 1 \pmod 8$, and $\Res_{Q(8)}(\lambda(f)) = 0 $ by \cite[Theorem 5.1(ii)]{Davis:1988}).
  This information about the normal map $f\colon M/G \to X$ was extracted 
  from the work of Madsen, Thomas and Wall (see \cite{madsen-thomas-wall1,madsen-thomas-wall2}). 
  
  Now we consider the  restriction of the surgery obstruction $\Res_H(\lambda^h(f))  \in L^h_3(\bZ H)$. 
 Since $X(H)$ is homotopy equivalent to an orthogonal space form, 
 $\Res_H(\lambda^h(f)) $ is the surgery obstruction of  a normal map between closed manifolds. 
 Therefore, $\Res_H(\lambda^h(f)) $ is detected by further restriction to the $2$-Sylow subgroup
  $C(4)$, and hence  $\Res_H(\lambda^h(f)) =0$. It follows that $H_1(M;\bZ)$ stably 
  supports a hyperbolic linking form, and hence 
  $$d = |H_1(M;\bZ)| \equiv r^2 \mod (8ab)$$
   is a square. Since $S(pq) = 0$, it follows that 
 $ \la r, \Sigma\ra = 0 \in \widetilde K_0(\ZG)$. Now by \cite[Theorem 5.1(ii)]{Davis:1988} 
 applied to $\lambda^h(f) \in L_3^h(\ZG)$, we see that $\lambda^h(f) =0$.  
 Therefore $G$ acts freely on an integral homology $3$-sphere.
\end{proof}

\begin{remark} By taking full advantage of Madsen's results as summarized in Theorem \ref{thm:res}, we could give 
a statement for the groups $Q(8a,b)$, under the assumption that $S(a',b') = 0$ for all divisors $1 \neq a' \mid a$, and $1\neq b' \mid b$. 
\end{remark}
\begin{remark}
We would  like to remove the almost linear $k$-invariant assumption.
 However, the group $G= Q(8)$
   acts freely and homologically trivially on $M$ with $H_1(M;\bZ) = \cy{3}$, since $Q(8) \times \cy 3$ acts
   freely on $\bS^3$. This action has non-linear $k$-invariant in our sense. Indeed, by the proof of Proposition \ref{prop:finiteness}, there is a degree 1 map $f\colon M/G \to X$, where $X$ is a Swan complex for $Q(8)$ with non-trivial finiteness obstruction. 
   \end{remark}

\section{Finite quotients of fundamental groups of 3-manifolds}

 In this section we consider closed 3--manifolds with  regular finite coverings 
which are rational homology spheres. The associated finite covering groups act freely on  such
rational homology $3$-spheres, so they afford examples to which our methods will apply. Note that according to \cite[Theorem 2.6]{Cooper:2000} every finite group in fact acts freely on some \emph{hyperbolic} (hence aspherical) closed rational homology sphere.  In such cases the fundamental group determines 
the topology, and we are really just considering finite index   normal subgroups of  certain Poincar\'e duality 
groups with vanishing first Betti number  (for background, see \cite[Chap.~VIII.10]{Brown:1982}).

Recall that for any group $Q$ and integer $n\ge 0$ we define 
the $n$--th term
of its \emph{derived series} as $Q^{(n+1)}= [Q^{(n)}, Q^{(n)}]$, where $Q^{(0)}=Q$. The derived series for a finite group stabilizes at a perfect normal subgroup, but may not terminate for an infinite group with non-zero abelianization. In fact an interesting open question is whether or not the derived series for
the fundamental group $\Gamma$ of a closed orientable hyperbolic $3$--manifold stabilizes 
if $\Gamma /\Gamma^{(n)}$ is finite and non--trivial for all $n>0$. If it does stabilize i.e. for some $i>0$, $\Gamma^{(i)}$ is perfect, then 
$\Gamma /\Gamma^{(i)}$ is a solvable group with periodic cohomology (of period dividing four), as it
acts freely on an integral homology 3--sphere, namely the covering space associated to the normal
finite index subgroup $\Gamma^{(i)}$. 
Independently of the stability question, one can ask (as in \cite{Cavendish:2015}) about
possible restrictions on the finite quotient groups $\Gamma/\Gamma^{(n)}$. 

Let $L$ denote a closed 3--manifold such that for some $n>0$ 
the quotient $\pi_1(L)/\pi_1(L)^{(n)}$ is finite. Let $\Gamma = \pi_1(L)$.
From the extensions
$$1\to \Gamma^{(i)}/\Gamma^{(i+1)}\to \Gamma/\Gamma^{(i+1)}\to \Gamma/\Gamma^{(i)}\to 1$$
for $0\le i\le n-1$, we infer that all the groups $\Gamma/\Gamma^{(i)}$ and
$\Gamma^{(i)}/\Gamma^{(i+1)}$ are finite in that range. Hence
the covering spaces $L_i$ corresponding to the $\Gamma^{(i)}$ are rational homology spheres. The finite groups $\Gamma/\Gamma^{(i)}$
act freely on them, with quotient $L$; note that  $H_1(L_i,\bZ)\cong \Gamma^{(i)}/\Gamma^{(i+1)}$.
Applying Corollary \ref{long-exact-sequence} we obtain

 \begin{proposition}\label{fundamental-groups}
Let $M$ denote a closed 3--manifold with $\Gamma = \pi_1(L)$ such that $\Gamma /\Gamma^{(n)}$ is finite
for some $n>0$. Then for all $0\le i\le n-1$ there are long exact sequences

\[
\dots\to\wH^{i+2}(\Gamma/\Gamma^{(i)},\bZ)\xrightarrow{\cup\,\sigma_i}  \wH^{i-2}(\Gamma/\Gamma^{(i)},\bZ)\to \wH^i(\Gamma/\Gamma^{(i)}, \Gamma^{(i)}/\Gamma^{(i+1)})\to \wH^{i+3}(\Gamma/\Gamma^{(i)},\bZ)
\to\dots
\]
These sequences are determined by elements $\sigma_i\in \wH^{-4}(\Gamma/\Gamma^{(i)},\bZ)$, which are
images of the respective generators in $\wH^0(\Gamma/\Gamma^{(i)},\bZ)\cong \cy{|\Gamma/\Gamma^{(i)}|}$.
\end{proposition}

\begin{corollary}
If $\Gamma^{(i)}$ is perfect, then $\sigma_i\in \wH^{-4}(\Gamma/\Gamma^{(i)},\bZ)$ is
an invertible element in the Tate cohomology of $\Gamma/\Gamma^{(i)}$.
\end{corollary}

As would be expected, Proposition \ref{fundamental-groups} can be used to obtain restrictions on the finite groups $\Gamma/\Gamma^{(i)}$.
As an application we take the opportunity to apply our methods to quickly sketch how to
obtain some of the results 
in \cite{Cavendish:2015} and \cite{Reznikov:1997}. 
 A closed hyperbolic 3--manifold $L$  is aspherical, and hence $H_3(L, \bZ) = H_3(\Gamma, \bZ)$ if  $\Gamma = \pi_1(L)$. For any finite quotient $G$ of $\Gamma$, we let $\sigma_G\in \wH^{-4}(G,\bZ)$ be
the element determined by the fundamental class of $L$, as in  Proposition \ref{prop:fundclass}.

\begin{proposition}[{Cavendish \cite{Cavendish:2015}}]
Let $L$ be a closed hyperbolic 3--manifold and $q\colon \Gamma\to G$, a surjective homomorphism
from $\Gamma = \pi_1(L)$ onto a finite group $G$ inducing an isomorphism
$H_1(\Gamma,\bZ)
\cong H_1(G,\bZ)$ and such that $ker(q)\subset \Gamma^{(2)}$.
Then the homomorphism $\phi_G\colon  \wH^2(G,\bZ)\to \wH^{-2}(G,\bZ)$ given
by $x\mapsto \sigma_G\cup x$ is bijective and cup product
defines a non--degenerate pairing
$\wH^2(G,\bZ)\otimes\wH^2(G,\bZ)\to \wH^4(G,\bZ)$.
Moreover, this pairing factors through a cyclic subgroup of 
$\wH^4(G,\bZ)$. 
\end{proposition}
\begin{proof}
First we observe that the condition $ker(q)\subset \Gamma^{(2)}$ implies that $G$ maps onto 
$\Gamma /\Gamma^{(2)}$, hence it is also finite and so the commutator $S =[\Gamma, \Gamma ]$ corresponds
to a covering $\tilde{L}$ of $L$ which is a rational
homology 3-sphere with a free action of the finite group $Q = H_1(G,\bZ)$. Now any group $G$ satisfying the
hypotheses necessarily maps onto $Q$, inducing an isomorphism $\wH^2(Q,\bZ)\cong \wH^2(G,\bZ)$
and so by naturality it will suffice to prove
the statements for $Q$. 

Consider the group extension $1\to S \to \Gamma \to Q\to 1$. Due to the vanishing of 
$H^1(S,\bZ)$ and $H^1(Q,\bZ)$ and
the fact that $H^3(\Gamma,\bZ)$ is torsion--free, the associated spectral sequence of the extension gives rise
to the following exact sequence in low degrees:
$$0\to H^2(Q,\bZ)\to H^2(\Gamma,\bZ)\to H^2(S,\bZ)^Q\to H^3(Q,\bZ)\to 0.$$
However, as $Q$ is the (finite) abelianization of $\Gamma$, we have $H^2(Q,\bZ)\cong H^2(\Gamma,\bZ)$
and so we infer that $\wH^0(Q, H^2(S,\bZ))\cong H^3(Q,\bZ)$ (it factors through Tate cohomology). 
On the other hand, from \eqref{long-exact-sequence} applied to the $Q$--action on $\tilde{L}$,
we have an exact sequence
$$0\to \text{coker}(\phi_Q)\to \wH^0(Q, H_1(S,\bZ))\to H^3(Q,\bZ)\to 0.$$
Noting the identification of $Q$--modules $H^2(S,\bZ)\cong H_1(S,\bZ)$, and using the isomorphism above,
we infer that $\phi_Q$ is surjective. As the domain and codomain of $\phi_Q$ have the same
number of elements this implies that it is an isomorphism.
Now given $y\in \wH^2(Q,\bZ)$ we can choose $z\in \wH^2(Q,\bZ)$ such
that the Tate dual $y^*=\sigma\cup z$. 
Then $0\ne z\cup y$ because $0\ne y^*\cup y=\sigma \cup z\cup y$,
showing that the pairing is nondegenerate.

Now let $J=[G,G]$, 
%$R=G/[G,G]$ and $S=[\pi_1(L),\pi_1(L)]=\pi_1(\tilde{L}/J)$.
$\Gamma$ maps onto $G$, so $S$ maps onto $J$ and $H_1(S,\bZ)$ maps onto $H_1(J,\bZ)$.
The condition $\ker(q)\subset \Gamma^{(2)}=S^{(1)}$ means that the kernel of
$q|_S\colon  S\to J$ is contained in $S^{(1)}$, and so the abelianized map is an isomorphism
$H_1(S,\bZ)\cong H_1(J,\bZ)$ or equivalently $H^2(S,\bZ)\cong H^2(J,\bZ)$.
%Given that $\tilde{L}/K$ is a rational homology sphere, $$H^2(\tilde{L}/J,\bZ)\cong H^2(J,\bZ)\cong H_1(J,\bZ).$$ 
Once again applying \eqref{long-exact-sequence} to the $Q$--action on $\tilde{L}$, we obtain the exact sequence
$$0 \to \wH^1(Q,H_1(S,\bZ)) \stackrel{d}{\to} \wH^4(Q,\bZ) \to \wH^0(Q,\bZ).$$
As has been observed in Remark \ref{differentials}, we can identify $d$ with the differential 
$d_3\colon  E_3^{1,2} \to E_3^{4,0}$ arising from the Serre spectral sequence for the fibration
$\tilde{L}\to L\to BQ$, which in this case is simply the Lyndon--Hochschild--Serre spectral sequence for the extension
$1\to S\to\Gamma\to Q\to 1$. Using the above identifications, this agrees with the corresponding
differential arising from the Lyndon--Hochschild--Serre spectral sequence for the group extension 
$1\to J\to G\to Q\to 1$.
Therefore the image of $d$
goes to zero under the inflation map $H^4(Q,\bZ)\to H^4(G,\bZ)$, and so it factors 
through $\coker~d \subset \wH^0(Q,\bZ)\cong \cy{|Q|}$, a cyclic group. 
Using the isomorphism $\wH^2(Q,\bZ)\cong \wH^2(G,\bZ)$ and naturality
of the cup product completes the proof.
\end{proof} 
%A quick proof of a result due to Reznikov \cite{Reznikov:1997}, followings from this, as in \cite{Cavendish:2015}.

\begin{corollary}[{Reznikov \cite{Reznikov:1997}}]
Let $L$ denote a closed three--manifold such that $G=\pi_1(L)/\pi_1(L)^{(n)}$ is a finite 2--group,
and $H_1(L,\bZ)\cong \cy{2} \times \cy{2}$. Then $G$ is a generalized quaternion group.
\end{corollary} 
\begin{proof}
Clearly $G/[G,G]\cong H_1(L, \bZ)\cong \cy{2}\times \cy{2}$ 
and so $G$ is a 2-group of maximal
class and thus must be either (generalized) quaternion, dihedral or semi-dihedral (see \cite{Gorenstein:1980}, Section 5.4). The condition
that the cup product pairing be non--singular eliminates the semi--dihedral groups 
(see \cite{Evens:1985}), 
and the fact that
the image has rank one eliminates the dihedral groups (see \cite{Handel:1993}). 
\end{proof}

%\bibliographystyle{ih}
%\bibliography{ah,ihmain}
%\end{document}
%%%%%%%%%%%%%%%%%%%%%%%%%%%%%
\providecommand{\bysame}{\leavevmode\hbox to3em{\hrulefill}\thinspace}
\providecommand{\MR}{\relax\ifhmode\unskip\space\fi MR }
\providecommand{\MRhref}[2]{%
  \href{http://www.ams.org/mathscinet-getitem?mr=#1}{#2}
}
\providecommand{\href}[2]{#2}

\end{document}